\documentclass[onefignum,onetabnum]{siamart190516}
\usepackage{mathtools}
\usepackage{graphicx}
\usepackage[caption=false]{subfig}
\usepackage{amssymb}
\usepackage{epstopdf}
\usepackage{bm}

\allowdisplaybreaks

\newcommand{\bu}{\mathbf{u}}

\newcommand{\balpha}{\bm{\alpha}}
\newcommand{\btheta}{\bm{\theta}}
\newcommand{\budget}{\mathcal{B}}

\title{Data-Driven Control of the COVID-19 Outbreak via Non-Pharmaceutical Interventions: A Geometric Programming Approach\thanks{Submitted to the editors on November 1, 2020.
\funding{This work was supported, in part, by the National Science Foundation under awards NSF-TRIPODS-1934876, CAREER-ECCS-1651433, NSF-III-200884556, and the Rockefeller Foundation.}}}

\author{Mikhail Hayhoe\thanks{Corresponding author. University of Pennsylvania, Philadelphia, PA (\email{mhayhoe@seas.upenn.edu}).}
\and Francisco Barreras\thanks{University of Pennsylvania, Philadelphia, PA (\email{fbarrer@sas.upenn.edu}, \email{preciado@seas.upenn.edu}).}
\and Victor M. Preciado\protect\footnotemark[3]}

\begin{document}
\maketitle
\begin{abstract}
In this paper we propose a data-driven model for the spread of SARS-CoV-2 and use it to design optimal control strategies of human-mobility restrictions that both curb the epidemic and minimize the economic costs associated with implementing non-pharmaceutical interventions. We develop an extension of the SEIR epidemic model that captures the effects of changes in human mobility on the spread of the disease. The parameters of our data-driven model are learned using a multitask learning approach that leverages both data on the number of deaths across a set of regions, and cellphone data on individuals' mobility patterns specific to each region. We propose an optimal control problem on this data-driven model with a tractable solution provided by geometric programming. The result of this framework is a mobility-based intervention strategy that curbs the spread of the epidemic while obeying a budget on the economic cost incurred. Furthermore, in the absence of a straightforward mapping from human mobility data to economic costs, we propose a practical method by which a budget on economic losses incurred may be chosen to eliminate excess deaths due to over-utilization of hospital resources. Our results are demonstrated with numerical simulations using real data from the Philadelphia metropolitan area.


\end{abstract}
\section{Introduction}
Ever since the first COVID-19 case was reported on December 31st 2019 \cite{sitrep_first}, the  SARS-CoV-2 pandemic has spread world-wide, reaching alarming levels of spread and severity \cite{sitrep_who}. The response to the first wave of COVID-19 by governments was the implementation of large scale non-pharmaceutical interventions (NPIs) ranging from contact tracing, quarantines and mask usage, to more aggressive measures like city wide shelter-in-place orders, air-travel restrictions and closures of non-essential businesses \cite{nbcstayhome}. In the absence of pharmaceutical treatment, prevention, or herd immunity, NPIs remain the only tool to curb the spread of the epidemic. At the same time, governments across the world have began implementing strategies to relax mobility restriction measures and reactivate the economy \cite{kaplan_frias_2020} while, at the same time, preventing the collapse of their healthcare systems. However, relaxing mobility restrictions too fast or carelessly can result in resounding waves, as we are currently observing for the case of COVID-19. In fact, as long as enough people in the population are susceptible, the danger of recurrent waves is not only real, but probable. In this situation, it is of utmost societal importance to develop reopening strategies in a principled manner utilizing the wealth of data readily available.

Several epidemic models have been proposed in the recent literature to simulate the effects of social distancing on the evolution of the pandemic; see, e.g., \cite{achterberg2020comparing, bhouri2020covid, chang2020mobility}. Although the majority of epidemic models in recent years are variations of the seminal mathematical models on theoretical epidemiology (see \cite{nowzari2016analysis} and references therein), the availability of rich datasets describing human mobility and behavior is rapidly changing the field of mathematical modelling of epidemics. Companies like Google, Foursquare, Safegraph, Baidu and others, have provided public access to massive datasets describing human mobility, enabling the development of data-driven epidemic models capturing the effects of mobility restrictions. Indeed, the choices faced by decision makers regarding disease management involve the use of multiple control actuations such as vaccination, quarantine, treatment or, as is the case for COVID-19, non-pharmaceutical interventions such as social distancing. These decisions must face the trade-off of minimizing the impact of the disease and the economic cost associated with the implementation of non-pharmaceutical interventions.  

In order to increase predictive power and utility for policy decision-makers, epidemic models have gradually increased their complexity to account for a multitude of features of real epidemics such as disease-specific compartmental models \cite{van2011gleamviz}, resurgence \cite{watts2005multiscale}, multi-scale effects \cite{hayhoe2018model}, seasonality \cite{balcan2010modeling}, differential risk structure in the population and healthcare system capacity \cite{ferguson2020report, lorch2020spatiotemporal}, among others. This increased sophistication in the modeling often comes at the cost of mathematical intractability, and most recommended interventions are heuristics based on simulations \cite{achterberg2020comparing, aleta2020modelling, balcan2010modeling,  birge2020controlling,chang2020mobility, ferguson2020report, lorch2020spatiotemporal}. Although informative for certain scenarios, these proposed interventions are not the result of rigorously formulated optimal control problems.

Conversely, the control of epidemics does not usually admit straightforward solutions from optimal control theory due to the presence of nonlinearities and/or the lack of convexity \cite{nowzari2016analysis}. Some important theoretical results formulate control problems as a static optimal resource allocation aiming to asymptotically drive the epidemic to extinction \cite{birge2020controlling, nowzari2015optimal, preciado2013optimal, preciado2014optimal, van2009virus, wang2003epidemic}. On the other hand, applications of Pontryagin's maximum principle (PMP) can be used to find exact solutions to resource allocation problems under some variations of the SIS and SEIR, for example in \cite{eshghi2015optimal, khouzani2011market, yan2008optimal}. Unfortunately, this approach does not easily generalize to other compartmental models, as solving the two-point boundary value problem that results from the PMP is intractable, in general, and its solution requires additional information about the structure of the optimal controller. Furthermore, the application of classical optimal control tools have found limited applicability in real epidemics since they cannot easily incorporate real data. 

A practical concern is whether it may be possible to design optimal control strategies based on mobility restrictionstaking in to account the impact on the economy. In this paper, we propose a data-driven model of the spread of COVID-19 and propose a data-driven optimal control problem that directly minimizes the number of predicted cumulative deaths by implementing mobility resctrictions in the population. We use real mobility data from Google \cite{googlemob} to learn a nonlinear mapping representing the impact of human mobility on the parameters of a dynamical epidemic model and propose a nonlinear, nonconvex optimal control problem that can be solved using tools from geometric programming \cite{boyd2007tutorial}.

Our model consists of an extension of the classic SEIR model, augmented with compartments for asymptomatic and hospitalized agents. We assume that the rate at which agents become infected in any given day is a function of the mobility trends in the population for that same day, reflecting the fact that an increase in mobility leads to more infections. We rely on data from Google's COVID-19 Community Mobility Reports \cite{googlemob} to capture the changes in visitation patterns to different Places of Interest (POIs). This mobility data consists of several time series measuring visits to various categories of places like Retail \& Recreation, Grocery \& Pharmacy, and Workplaces. The dataset is organized into separate time series for all counties in the United States, and measures visits to multiple categories of places against a benchmark established in January and February of 2020. The key point is that a decision maker can enforce restrictions on the number visits to each of these categories to reduce the spread of the epidemic while incurring a cost to the economy. Our objective is then to design an optimal strategy to contain the spread of COVID-19 while taking in to account the economic cost associated with these mobility restrictions.

The structure of the paper is as follows. In \Cref{sec:background} we introduce the notation used as well as some necessary background in geometric programming. In \Cref{model} we discuss the specifics of our data-driven model, consisting of a mobility layer and an epidemic layer. In \Cref{sec:learning} we discuss the details of our learning strategy to identify the parameters of the model. In \Cref{sec:optimal} we present our optimal control framework and present simulations showing the effectiveness of our method. In \Cref{sec:conclusion} we conclude and discuss possibilities for further research.

\section{Background and Notation}
\label{sec:background}
Throughout this paper bold characters are used to denote vectors and upper-case characters denote either matrices or compartments of the epidemic model. For the following definitions let $x_1, \dots, x_n \geq 0$ denote $n$ non-negative variables, and let $\mathbf{x} = (x_1, \dots, x_n)$. 

\begin{definition}[Monomial]
A function $f(\mathbf{x})$ is called a \textit{monomial} if it has the form
\begin{align*}
    f(\mathbf{x}) = c x_1^{a_1}x_2^{a_2}\dots x_n^{a_n},
\end{align*}
for $a_1, \dots, a_n \in \mathbb{R}$ and $c>0$.
\end{definition}
\begin{definition}[Posynomial]
A sum of one or more monomials is called a \textit{posynomial}, that is, a function of the form
\begin{align*}
    f(\mathbf{x}) = \sum_{i=1}^k c_i x_1^{a_{i,1}} x_2^{a_{i,2}}\dots x_n^{a_{i, n}}.
\end{align*}
\end{definition}
Since posynomials admit negative exponents but do not admit negative coefficients they are not necessarily polynomials, and vice versa. We remark that posynomials are closed under addition, multiplication, and positive scalar multiplication. This implies that if the entries of two matrices $A \in \mathbb{R}^{m\times k}$ and $B\in \mathbb{R}^{k\times n}$ are posynomials of the same variables, then so are the entries of their product $AB$, since $[AB]_{i,j} = \sum_{t=1}^k A_{i,t}B_{t,j}$, which is a sum of products of posynomials. This result extends trivially to the product of an arbitrary number of matrices with posynomial entries.

\begin{definition}[Convex in log-scale]
A function $f(\mathbf{x})$ is convex in log-scale if the function $F(y) := \log f(\exp (\mathbf{y}))$ is convex (where the exponentiation is component wise).
\end{definition}
A careful application of Hölder's inequality shows that posynomials are convex in log-scale.

We solve the epidemic control problems presented herein using a quasi-convex optimization framework called \textit{geometric programming} \cite{boyd2007tutorial, boyd2004convex}, which has found wide applicability in fields such as communication systems \cite{chiang2005geometric}, epidemiology \cite{preciado2014optimal}, and control \cite{ogura2019geometric}, among others. A \textit{geometric program} (GP) is a mathematical optimization program of the form
\begin{align*}
        \underset{\mathbf{x}}{\text{minimize}}\quad& f(\mathbf{x})\\
    \text{subject to}\quad& q_i(\mathbf{x}) \leq 1 \qquad \text{for } i \in 1, \dots m\\
                     & h_i(\mathbf{x}) = 1 \qquad \text{for } i \in 1, \dots p,    
\end{align*}
\noindent where $f(\mathbf{x}), q_1(\mathbf{x}), \dots, q_m(\mathbf{x})$ are posynomials and $h_1(\mathbf{x}), \dots, h_p(\mathbf{x})$ are monomials. Due to the convexity in log-scale, one can reduce a GP to a convex program by means of the logarithmic change of variables $y_i = \log (x_i)$ and transforming the objective and constraints with the logarithmic transformations $F(\mathbf{y}) = \log f(\exp (\mathbf{y}))$, $Q_i(\mathbf{y}) = \log q_i(\exp (\mathbf{y}))$ and $H(\mathbf{y}) = \log h(\exp (\mathbf{y}))$ to obtain
\begin{align*}
        \underset{\mathbf{y}}{\text{minimize} }\quad& F(\mathbf{y})\\
    \text{subject to}\quad& Q_i(\mathbf{y}) \leq 0 \qquad \text{for } i \in 1, \dots m\\
                     & H_i(\mathbf{y}) = 0 \qquad \text{for } i \in 1, \dots p,
\end{align*}
which is a convex program that can be efficiently solved using, for example, primal-dual interior-point methods; see \cite{dahl2019primal} for more details.

\section{Model} \label{model}

We now describe the epidemiological model under analysis. We consider a region with $N$ individuals and propose a population model with \emph{homogeneous mixing}, i.e.,  every pair of individuals come in contact with a probability that depends on aggregated mobility variables. Population models are commonly used in the absence of granular data on the network of social contacts in a region \cite{nowzari2016analysis}. We assume that each individual can visit POIs belonging to different categories $1, 2, \dots, K$; let $\mathbf{m}(t)= (x_1(t), m_2(t), \dots, m_K(t))$ denote a vector of human mobility variables capturing the volume of visits to each of those categories at a particular discrete time $t$ (e.g., days). In particular, we use publicly available mobility data from Google's COVID-19 Community Mobility Reports \cite{googlemob} which capture daily changes in visitation patterns to public places, for example, Retail \& recreation, Grocery \& pharmacy, and Parks, as well as time spent at work. For each category, Google reports the relative change in visits compared to a baseline of mobility measured before the lockdown measures took place. This baseline corresponds to the median daily visits to each category over the period comprising January 3 through February 6, 2020.

Our model consists of two layers: a mobility layer and an epidemic layer.

The mobility layer captures the effect of human mobility on the spread of the disease and influences the dynamics taking place on the epidemic layer. In the epidemic layer, we consider an extension of the classic SEIR epidemic model \cite{brauer2019mathematical, martcheva2015introduction} which explicitly accounts for asymptomatic hosts and hospitalizations. Each of the individuals in a region belongs to one of seven possible compartments described below. In this model, $S(t)$ represents the number of individuals susceptible to becoming infected at a discrete time $t$. In our optimal control problems, we will consider a finite horizon $T$ over which we can assume an almost constant number of susceptible individuals; hence, we assume $S(t) \approx S_0$ for all $0 \leq t \leq T$. The variable $E(t)$ represents the number of individuals who have contracted the virus (exposed) but are in an incubation period at time $t$. After the incubation period, agents can move to one of the infectious compartments; $I(t)$ represents the number of symptomatic individuals and $A(t)$ represents the number of asymptomatic individuals. The asymptomatic compartment is included since asymptomatic individuals play a crucial role in the spread of COVID-19, with transmission rates that are different from symptomatic individuals \cite{bai2020presumed, gandhi2020asymptomatic}. Asymptomatic individuals eventually recover on their own and move on to the recovered compartment, represented by the variable $R(t)$. Symptomatic individuals can recover on their own, or their symptoms can worsen and they subsequently require hospitalization, in which case they are moved into a hospitalized compartment, represented by the variable $H(t)$. Since hospital capacity is a principal concern with the treatment of COVID-19, we include this compartment to constrain the control problems described in \Cref{sec:optimal}. In particular, our mobility-based control input will be constrained to prevent a hospital capacity overflow. Finally, agents that are hospitalized may either recover and transition to the compartment represented by $R(t)$ or may die, and subsequently transition to the compartment represented by $D(t)$. With explicit data on the number of deaths in every region, we may train our model to predict the population of this compartment. We make the simplifying assumption that only individuals with severe symptoms are at a risk of dying and, hence, all of them are eventually hospitalized. 

\begin{figure}[ht!]
    \centering
    \includegraphics{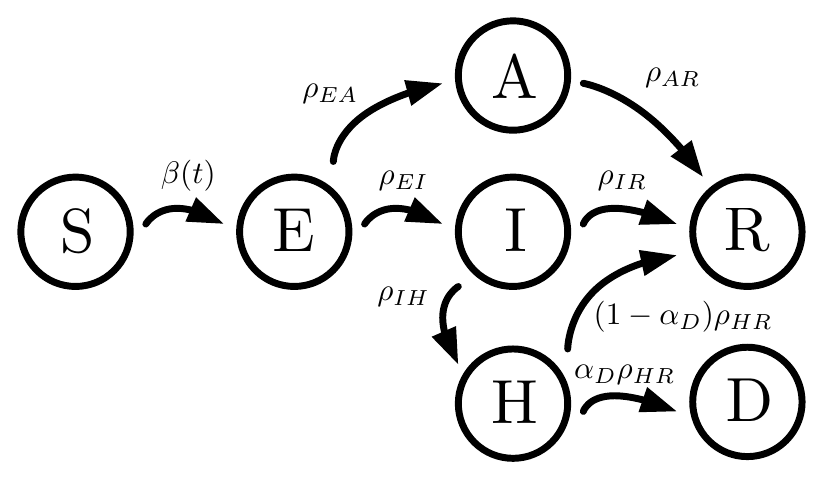}
    \caption{Illustration of the epidemic model under consideration.}
    \label{fig:seird}
\end{figure}

All parameters related to the dynamics of this model are summarized in \cref{tab:parameters}. Using these parameters, the discrete-time evolution of the number of individuals in each compartment, illustrated in \cref{fig:seird}, is given by:
\begin{align}
    E(t+1) &= \left(1-\rho_{EI}-\rho_{EA}\right)E\left(t\right) + S_{0}\beta\left(t\right)(\gamma_{A}A\left(t\right) + I\left(t\right)), \label{eq:dyn_1}\\ 
    I(t+1) &= \left(1-\rho_{IR}-\rho_{IH}\right)I\left(t\right)+\rho_{EI}E\left(t\right), \label{eq:dyn_2}\\
    A(t+1) &= \left(1-\rho_{AR}\right)A\left(t\right)+\rho_{EA}E\left(t\right), \label{eq:dyn_3}\\
    H(t+1) &= \left(1-\rho_{HR}-\rho_{HD}\right)H\left(t\right)+\rho_{IH}I\left(t\right), \label{eq:dyn_4} \\
    R(t+1) &= R(t) + \rho_{IR}I(t) + \rho_{AR}A(t) + (1 - \alpha_D)\rho_{HR}H(t). \\
    D(t+1) &= D(t) + \alpha_D\rho_{HR}H(t) \label{eq:D_t}
\end{align}

In our model, we assume that susceptible individuals can transition into the exposed compartment when in contact with either Infected (symptomatic) or Asymptomatic individuals. We assume that the rate at which asymptomatic individuals infect others is weighted by a constant (unknown) parameter $\gamma_A$. We also assume that the portion of hospitalized individuals who die, relative to those that recover, is equal to an unknown constant $\alpha_D$. In \Cref{sec:learning}, we will introduce a methodology to learn these (and other) unknown parameters in our model. As mentioned above, this model is intended to solve optimal control problems over a finite time horizon $T$ over which we can assume that the number of susceptible individuals at any time $0 \le t\le T$, $S(t)$ is well approximated with a constant, hence, we set $S(t) = S_0$; hence, we omit the dynamical equation corresponding to $S(t)$. Moreover, this assumption linearizes the dynamics of the states. As we will see in \Cref{sec:optimal}, this linearization ensures that the entries of the state vector at any given time are posynomials on the parameter $\beta(t)$. However, we incorporate a non-linear dependency of the parameter $\beta(t)$ on the mobility restriction variables, which we will use as our external control variable, rendering the resulting model non-linear and multiplicative in the control input.

\begin{table}[ht!]
\centering

\begin{tabular}{ll}
\hline
Parameter  & Description                                                                           \\ \hline \hline
$\beta (t)$ & Rate at which susceptible individuals become infected due to \\
    & contacts with infectious individuals at  time $t$;  $\beta(t)$ is a function \\
    & of the mobility variables $\mathbf{m}(t)$ at time $t$   \\
$\gamma_A$  & Weight representing lower risk of infection when infectious \\
    & individuals are asymptomatic \\
$\rho_{EI}$ & Rate at which exposed individuals become symptomatic\\
$\rho_{EA}$ & Rate at which exposed individuals become asymptomatic  \\
$\rho_{IR}$ & Rate at which symptomatic individuals recover on their own  \\
$\rho_{IH}$ & Rate at which symptomatic individuals develop severe symptoms\\
            &  and become hospitalized\\
$\rho_{AR}$ & Rate at which asymptomatic individuals recover on their own \\
$\rho_{HR}$ & Rate at which hospitalized individuals recover     \\
$\alpha_D$ &Proportion of hospitalized individuals that die, relative to those \\
    & that recover \\
\hline \hline                                        
\end{tabular}

\caption{Summary of parameters in epidemic model.}
\label{tab:parameters}
\end{table}

The mobility layer of our model incorporates the effects of non-pharmaceutical interventions, such as social distancing and other forms of mobility restrictions, which a decision maker may employ to curb the spread of the epidemic. By reducing human mobility, a decision maker induces fewer contacts between susceptible and infected individuals and, thus, reduces the risk of infection. In particular, we relate the infection rate $\beta(t)$ to a time series $\mathbf{m}(t)$ of human mobility variables by means of an unknown function $f(\cdot)$. We choose $f$ to be a parametric function whose parameters will be learned from data (described in detail in \Cref{subsec:learn_mobility}).

In order to employ non-pharmaceutical interventions, a decision maker designs a mobility control strategy to set the human mobility variables $\mathbf{m}(t)$ for some finite horizon $t \in \{0,\ldots,T\}$. In mathematical terms, the decision maker designs an input $\{\bu(t)\}_{t=0}^T$ which affects future values of the mobility variables. For simplicity, we assume an identity mapping between mobility and the input, so that $\mathbf{m}(t) = \bu(t)$, and $\bu(t)$ is in some set of admissible actions $\mathcal{U}$.

Intuitively, lower values of $\bu(t)$ correspond to more restrictions on human mobility. The decision maker may have fine-grained control over their control strategy, for example by closing individual establishments, imposing occupancy limits, or restricting hours of operation, and as such we treat the individual components of the control action $u_k(t)$ as continuous variables. Moreover, some categories may have different admissible actions, e.g., it may not be possible to close down all pharmacies but closing all gyms is reasonable. In mathematical terms, we will consider a set of allowable control actions $\mathcal{U}$ is that it be described using posynomial inequalities and monomial equalities. 

Unfortunately, implementing mobility restrictions in this manner cannot be done without incurring a financial loss. Closure of businesses causes economic losses, which need to be taken into consideration when selecting an appropriate control strategy. In particular, applying the temporal control strategy $\bu(t)$ of mobility restrictions incurs a cost $C(\bu(t))$ which we assume to be monotonically decreasing with $\bu$ and convex in log-scale, reflecting that the costs on society of restricting mobility are marginally increasing.

\section{Learning the parameters from data} \label{sec:learning}

\begin{figure}[ht!]
    \centering
    \includegraphics[width=\linewidth]{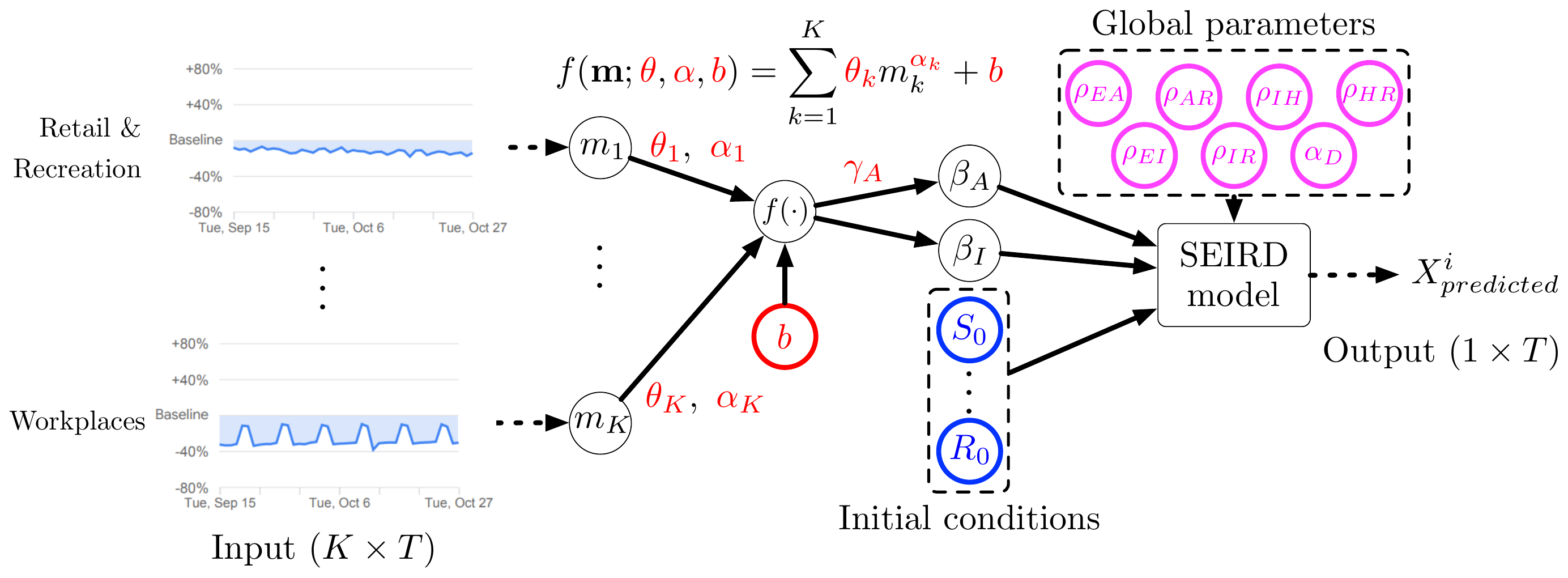}
    \caption{The learning pipeline for a given region. Global parameters, learned using all available data, are shown in magenta; initial conditions, learned for each region, are shown in blue; mobility mapping parameters, again different across regions, are shown in red.}
    \label{fig:pipeline}
\end{figure}

Several recent epidemic prediction methods opt to set some or all of the parameters in their models to estimations from the medical and virology literature (e.g. \cite{achterberg2020comparing, aleta2020modelling, bhouri2020covid, birge2020controlling, chang2020mobility}). However, these parameters often have wide confidence intervals and are commonly inferred from statistical models that do not take in to account the effects of social distancing and hospital capacity \cite{ali2020serial}. In contrast to these approaches, our model is entirely data-driven in that all parameters (including initial conditions) used are learned directly from data. In particular, we employ a multitask learning approach \cite{caruana1997multitask} by leveraging both data on the number of deaths across a set of regions, and mobility data describing how often individuals in a region visit different points of interest. Indeed, certain parameters of the epidemic model are intrinsic to the disease; hence, they do not depend on the geographical location from which data is collected. As such, when calibrating this model to a given region, we can benefit from the data from other regions by employing a multitask learning framework, providing better parameter estimates and avoiding overfitting. Towards this goal, we pool data from multiple regions (e.g., US counties) and minimize a global cost function in which the global parameters are shared across regions but the mobility parameters and initial conditions are specific to each region. The learning pipeline from data to predictions for a single region is shown in \cref{fig:pipeline}. We fit our models using publicly available mobility data from Google's COVID-19 Community Mobility Reports \cite{googlemob} which captures daily changes in visitation patterns to public places, for example, Retail \& Recreation, Grocery \& Pharmacy, and Parks, as well as time spent at Workplaces. This dataset is organized as different time series for six different categories. For each category, Google reports the relative change in visits compared to a baseline of mobility measured before the lockdown measures took place. This baseline corresponds to the median daily visits to each category over the period comprising January 3 through February 6, 2020. Furthermore, we use public data from The New York Times, based on reports from state and local health agencies \cite{nytimes}, consisting of daily and cumulative caseloads and deaths attributed to COVID-19 in The United States. To account for inconsistencies and lags in reporting, we compute a 7-day rolling average on the original time series for the calibration of our model.

As mentioned above, there are two layers to our model, namely (1) the \emph{mobility layer}, which is a mapping from mobility data to the infection rate $\beta(t)$ (described in \cref{subsec:learn_mobility}) and (2) the \emph{epidemic layer}, which describes the dynamics of the disease itself (discussed in \cref{subsec:learn_model}). Since parameters such as the latency period, ratio of infected individuals who develop symptoms, and case fatality ratio are intrinsic to the disease and should not vary greatly based on the geographical area being studied, we group these together across regions as global parameters and learn them jointly with all the available data. However, the  mapping from mobility data to the infection rate and initial conditions of the regional epidemic are dependent on the locality, and thus they are learned using only the data from their region.

\subsection{The mobility layer}\label{subsec:learn_mobility}

To learn the function $f: \mathbf{m} \mapsto \beta$ from mobility data to the infection rate, we must first select an appropriate class of functions for such a mapping. Although we could use any parametric familt of functions, such as neural networks, to estimate $f$, not all choices are tractable. In particular, neural networks may provide great prediction performance but would render an intractable control problem. In order to obtain a tractable control problem, we chose to model the function $f$ using a parametric posynomial function \cite{boyd2007tutorial}. As we will show in \Cref{sec:optimal}, this choice allows us to use geometric programming to efficiently solve several optimal control problems of interest. Thus, we model $f$ in a parametric way as
\begin{align}\label{eq:mobility_map}
    f(\mathbf{m}; \btheta, \balpha, b) = \sum_{k=1}^K \theta_k m_k^{\alpha_k} + b.
\end{align}
From a practical standpoint, this posynomial approximation is justified because $\beta$ can be viewed as the product of the contact rate (the expected number of contacts an individual has with others) and the transmission risk, which is constant over time. Moreover, the number of contacts within a category should exponentially increase with the number of visits to points of interest in that category. Since the mobility data is stratified across $K$ different categories, we allow the parameters to be different across the categories. Thus, in the parametric function $f(\mathbf{m}; \btheta, \balpha, b)$ the probability of transmission is captured by $\btheta = (\theta_1,\ldots,\theta_K) \in \mathbb{R}_{\geq 0}^K$, the exponential growth of infectivity is captured by $\balpha = (\alpha_1,\ldots,\alpha_K) \in \mathbb{R}^K$, and the bias term $b$ accounts for potentially unmodeled infections.

Since susceptible individuals may become infected by either symptomatic or asymptomatic individuals, the mobility mapping $f$ is incorporated into the epidemic layer in two terms, as seen in \cref{eq:dyn_1}. Firstly, it is used to model new infections from symptomatic individuals via the term $\beta(\mathbf{m}(t))S_0 I(t)$; secondly, we include a weighting term $\gamma_A$ in the term $\gamma_A \beta(\mathbf{m}(t)) S_0 A(t)$ to model the rate of new infections from asymptomatic infectious individuals. Thus, altogether we denote the set of parameters corresponding to the mobility mapping function for region $i$ as $\Psi_{mobility}^i \coloneqq \{\btheta, \balpha, b, \gamma_A\}$.

\subsection{The epidemic layer}\label{subsec:learn_model}
Our model is a latent-space model; hence, the states are not fully observable. Following common practice with these models, we treat the unobserved initial conditions as unknown parameters to be identified from the data. Since the dynamical trajectories of the epidemic are different across regions (e.g., US counties), for each region $i$ we learn a set of initial conditions $\Psi_{0}^i \coloneqq \{S_0^i,\ldots,R_0^i\}$. As mentioned previously, we follow a multitask learning approach and, thus, the remaining global parameters, which we assume to be intrinsic to the disease, are shared across all regions and learned collectively using data available from all regions. The set of global parameters is denoted by $\Psi_{global} \coloneqq \left\{\rho_{EA}, \rho_{EI}, \rho_{AR}, \rho_{IH}, \rho_{IR}, \rho_{HR}, \alpha_D\right\}$, which includes all clinical parameters that depend on the nature of the virus alone (i.e., they are not influenced by social mobility).

\subsection{Simulated predictions for Philadelphia and surrounding counties}\label{subsec:learn_simulations}

\begin{figure}[ht!]
    \centering
    \subfloat[Incident deaths in Baltimore County, MD]{\includegraphics[width=0.31\linewidth]{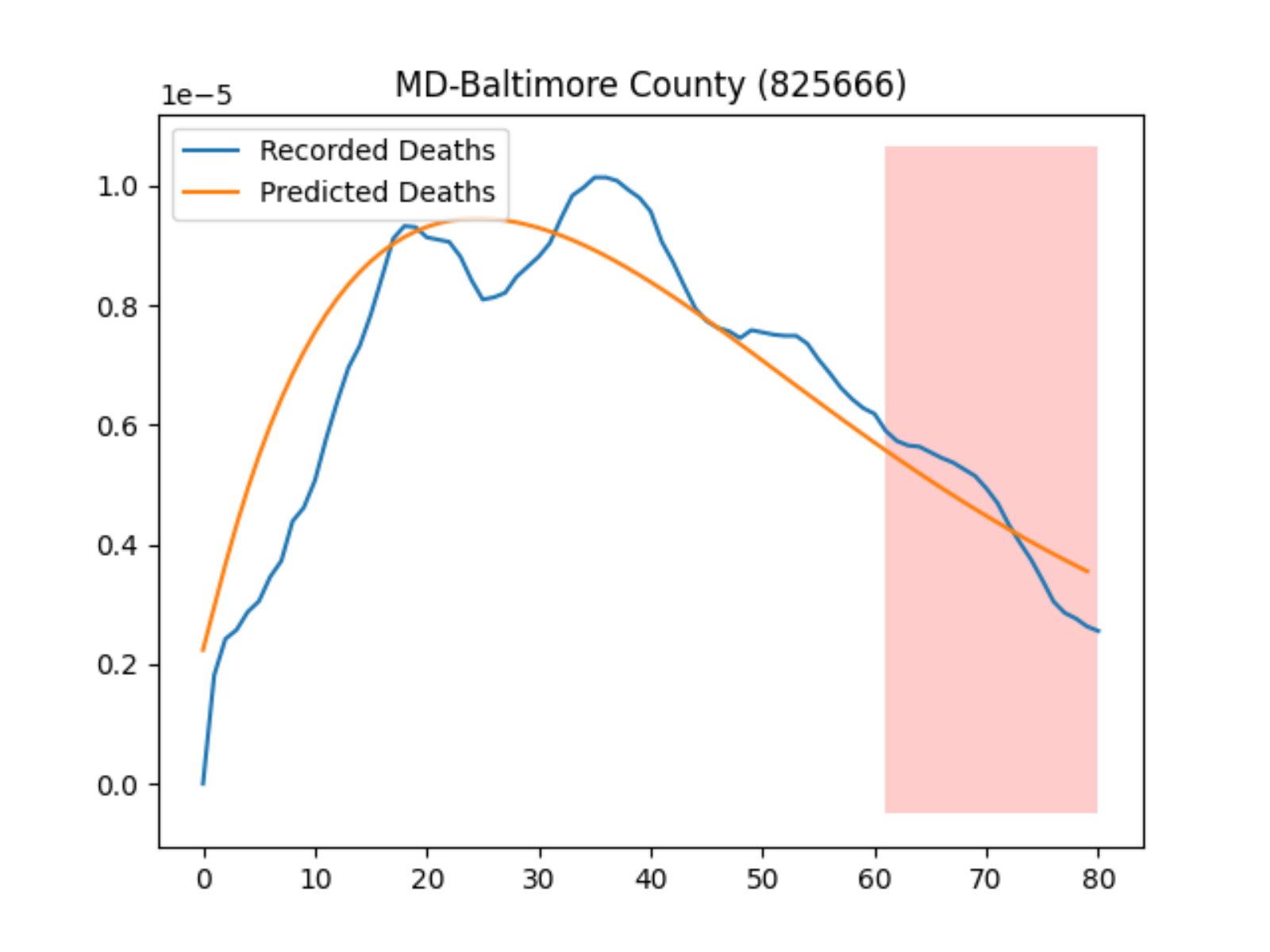}}
    \quad
    \subfloat[Incident deaths in Montgomery County, PA]{\includegraphics[width=0.31\linewidth]{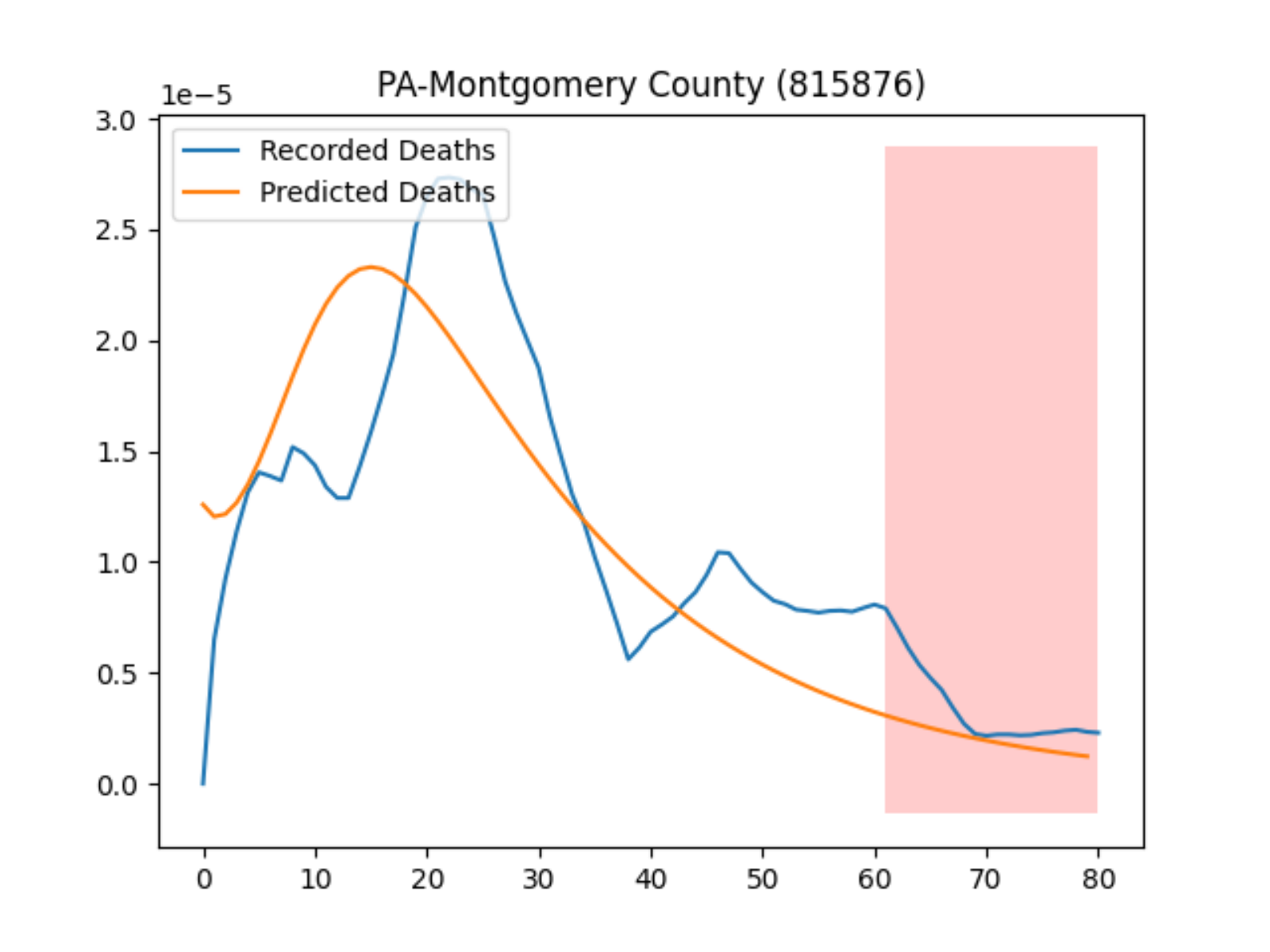}}
    \quad
    \subfloat[Incident deaths in Philadelphia County, PA]{\includegraphics[width=0.31\linewidth]{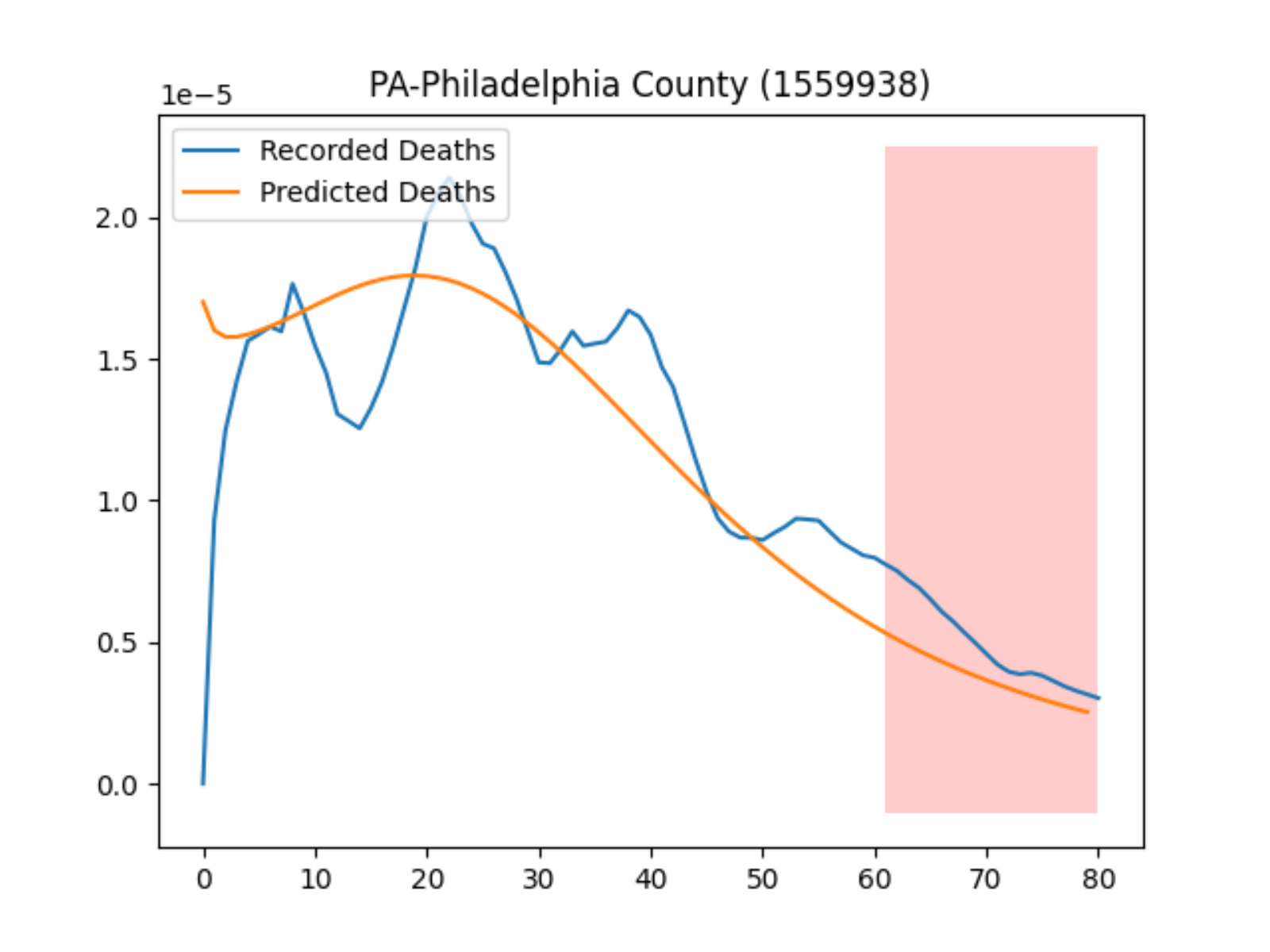}}
    \\
    \subfloat[Cumulative deaths in Baltimore County, MD]{\includegraphics[width=0.31\linewidth]{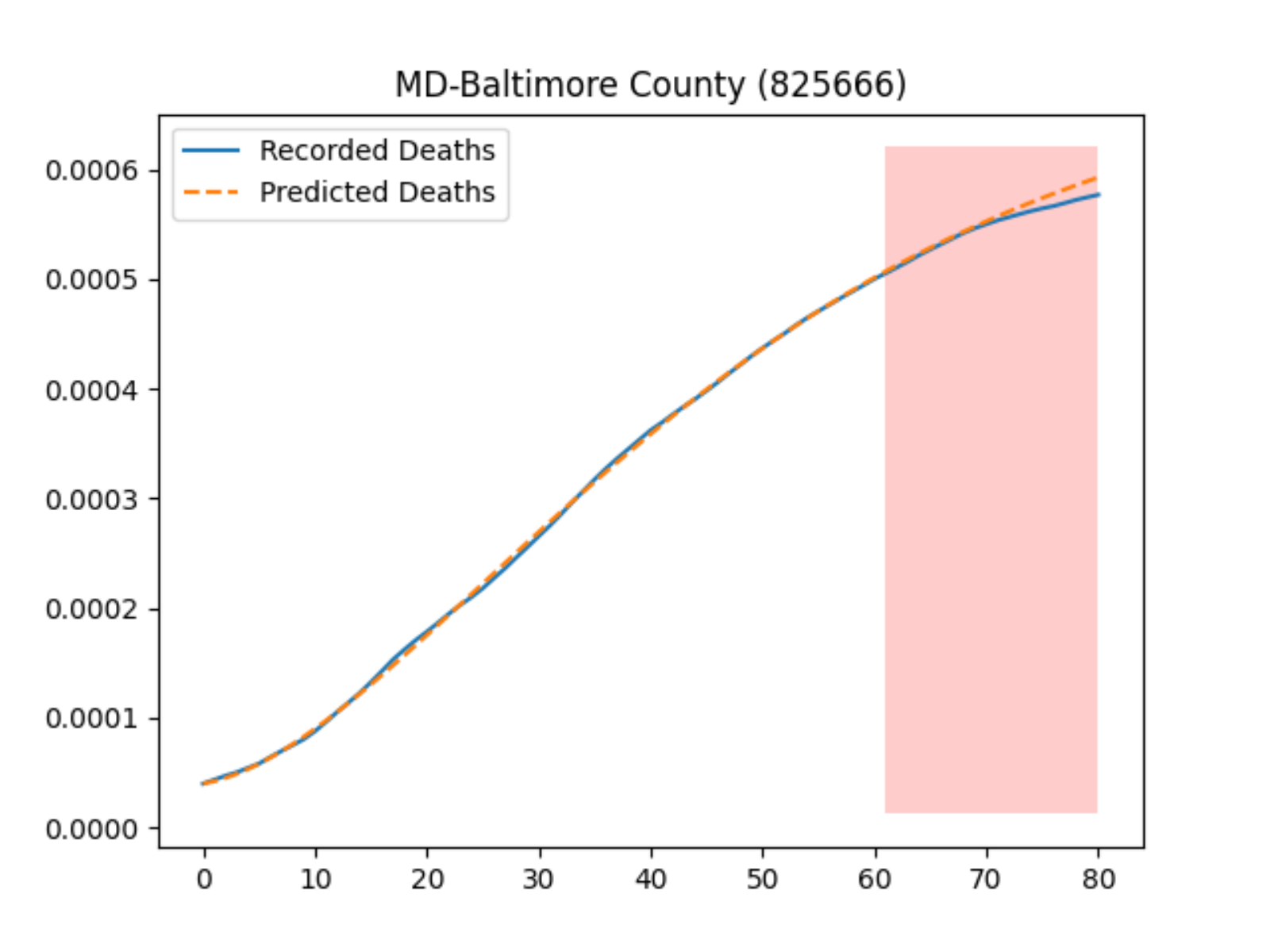}}
    \quad
    \subfloat[Cumulative deaths in Montgomery County, PA]{\includegraphics[width=0.31\linewidth]{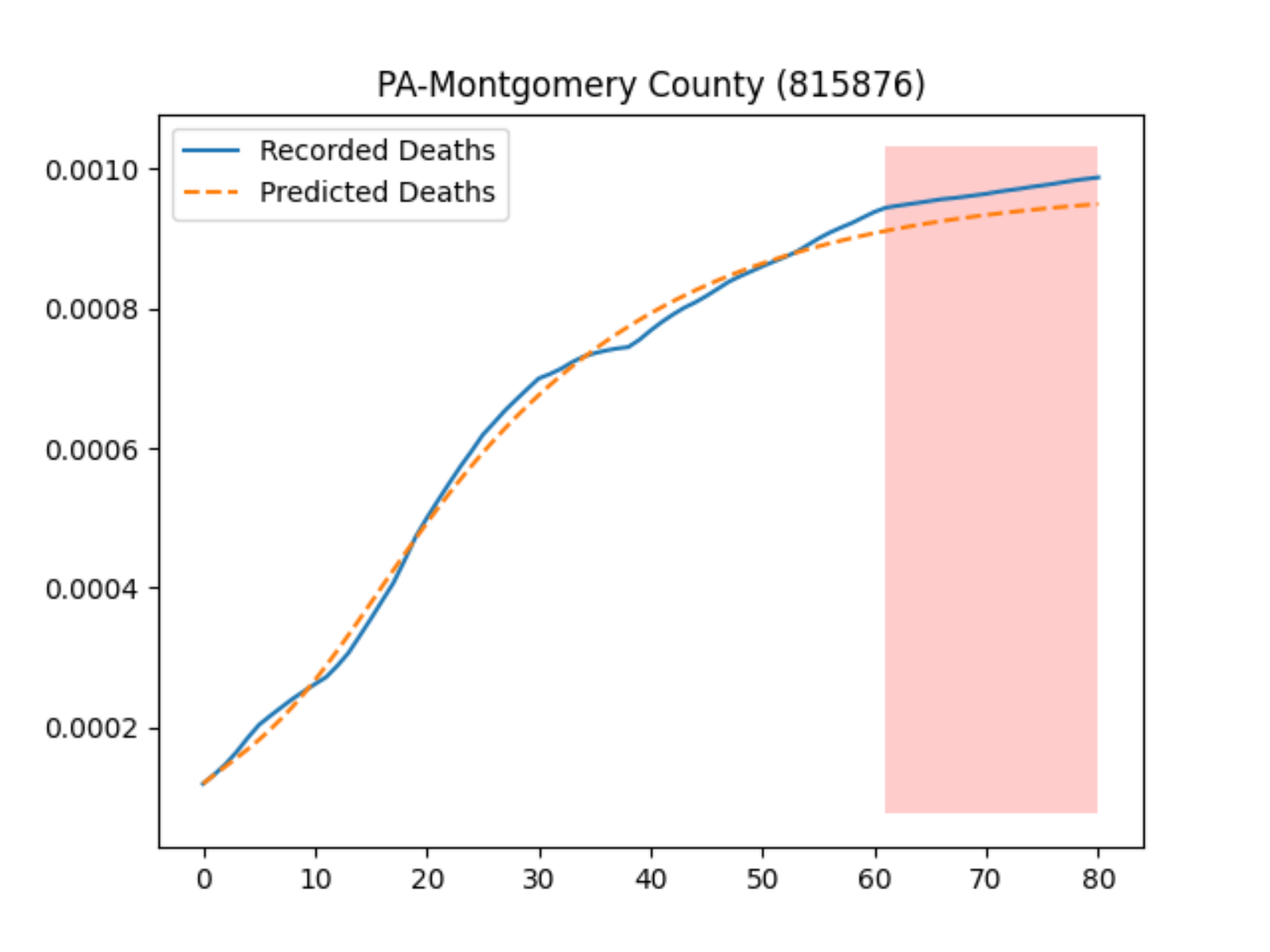}}
    \quad
    \subfloat[Cumulative deaths in Philadelphia County, PA]{\includegraphics[width=0.31\linewidth]{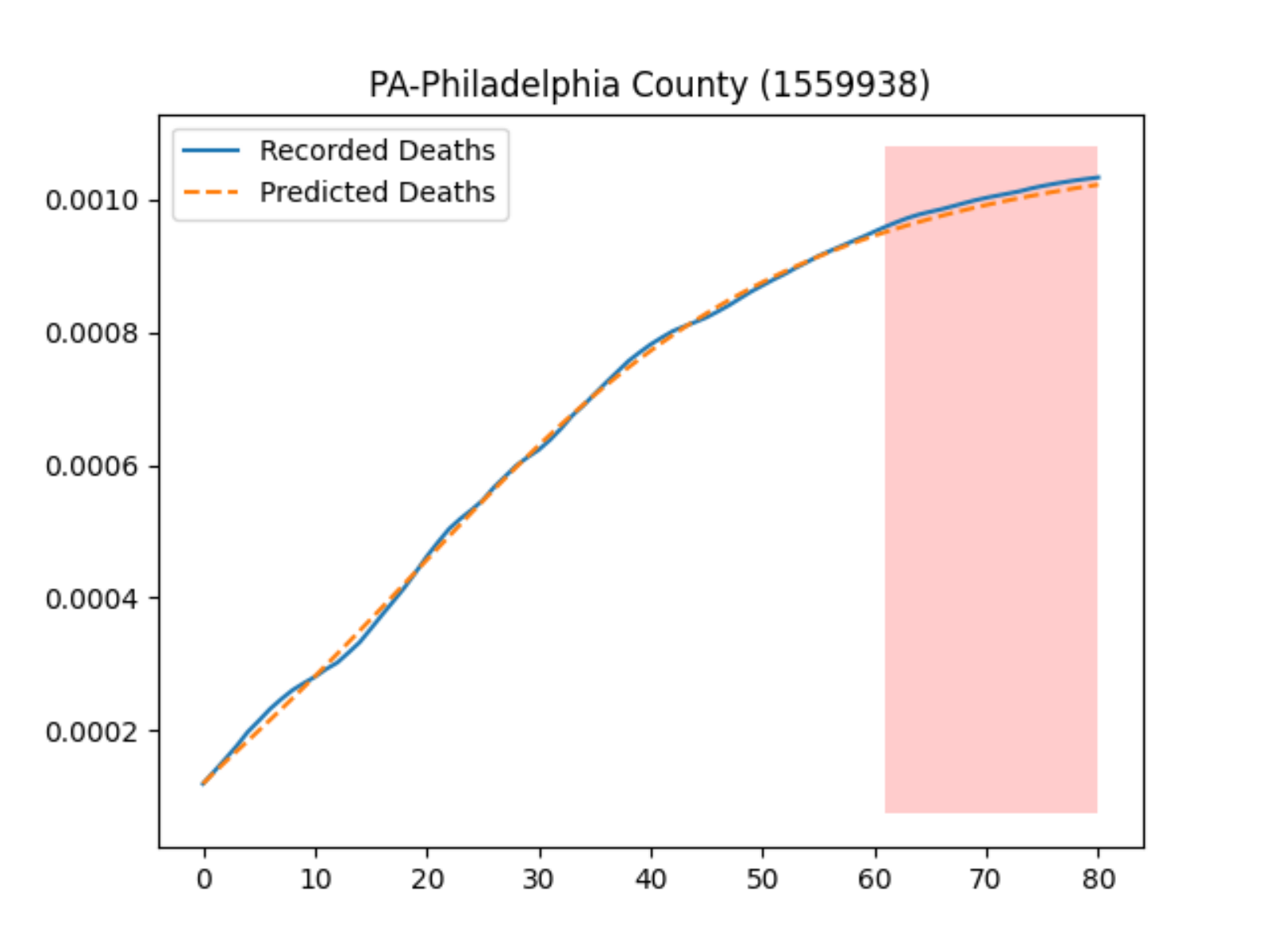}}
    \caption{Example of predicted deaths for several US counties, trained on 60 days of mobility and death count data. Rolling 7-day averages of real data are shown in blue; predictions are shown in orange. The area in white denotes training data, and the area in red denotes test data.}
    \label{fig:predictions}
\end{figure}

In order to validate the predictive accuracy of our data-driven model, we conducted a case study using counties from the Delaware Valley metropolitan statistical area (commonly known as the \textit{Philadelphia metropolitan area}). Using county-level data, we learned both the local mobility mapping functions and initial conditions, as well as the global clinical parameters of the epidemic. Due to the known inconsistencies and lags in reporting of cases and deaths, we use the rolling 7-day average of cumulative deaths for both training and prediction. Formally, if we have training data for $M$ counties over $T$ days, the training loss for the set of parameters $\Psi \coloneqq \Psi_{global} \cup_{i=1}^M \Psi_{mobility}^i \cup_{i=1}^M \Psi_{0}^i$ is the following mean-squared error loss function:
\begin{align}\label{eq:train_loss}
    \ell_{train}(\Psi) = \frac{1}{MT}\sum_{i=1}^M\sum_{t=1}^T \left(\frac{X^i(t)}{N_i} - \frac{\hat{X}^i(t; \Psi)}{N_i}\right)^2,
\end{align}
where $N_i$ is the population of region $i$, $X^i(t)$ is the measured rolling 7-day average of cumulative deaths on day $t$ for region $i$, and $\hat{X}^i(t; \Psi)$ is the predicted value for the same region for a set of parameters $\Psi$. We normalize the number of deaths by the population in each county to avoid biasing our predictions towards counties with a larger population.

Our model was trained by computing gradients of the loss function $\ell_{train}$ with respect to all parameters in $\Psi$ via the automatic differentiation package \texttt{autograd} \cite{autograd} and running stochastic gradient descent using \texttt{adam} \cite{adam} over several independent trials using different initial guesses to account for the non-convexity of the training problem. Global clinical parameters were initialized by using plausible values from the medical literature \cite{day2020covid, nishiura2020estimation, lauer2020incubation, he2020temporal, bhouri2020covid, rahmandad2020estimating, pei2020differential, woelfel2020clinical}, while local parameters are initialized randomly in each trial. Different counties are split across random batches, allowing the global parameters (trained using all data in our multitask approach) to converge more quickly, which in turn allows the local parameters to converge to values which fit more closely with the global parameters. From these trials, we select the set of parameters with lowest testing error, and present the predictions using these chosen parameters in \cref{fig:predictions}.

\section{Optimal control using geometric programming} \label{sec:optimal}

Traditional optimal control techniques are not directly applicable to compartmental epidemic models for a number of reasons. First, epidemic models are typically nonlinear and the effect of non-pharmaceutical interventions is not additive, but multiplicative. Therefore, standard techniques such as LQR cannot be readily employed. Furthermore, a direct application of Pontryagin's Maximum Principle results in a high-dimensional two-point boundary value problem for which numerical methods have no convergence guarantees and present scalability issues. Due to the modeling choices proposed in this paper, we obtain a mobility-driven epidemic dynamics amenable to solve certain optimal control problems using tools from geometric programming \cite{boyd2007tutorial}. In particular, we can solve the data-driven optimal control problems aiming to minimize the final number of deaths while respecting budget constraints on the economic costs associated with implementing mobility restrictions, as well as avoiding hospital overflows, with guarantees of global optimality.

We start our analysis by considering the minimization of the final number of deaths. For completeness in our analysis,we allow the inclusion of a daily discount factor $\gamma_D$ and a terminal cost $\gamma_{\infty}$ on the number of deaths at the end of the time horizon $T$. Hence, the decision maker aims to minimize the following objective function:
\begin{align}\label{eq:J}
J \coloneqq \sum_{t=1}^{T-1} \gamma_D^t D(t) + \gamma_{\infty}D(T).    
\end{align}
The discount factor $\gamma_D \in (0,1)$ is motivated by the uncertainty of deaths in the future which might be prevented by interventions not available in the present day. For example, the probability that a vaccine is developed in the future increases as time passes. In particular, $1 - \gamma_D$ can be considered the probability of a vaccine being created on each day in the future, so the probability of no vaccine being widely available by day $t$ is $(\gamma_D)^t$ and, hence, deaths predicted at day $t$ are only accounted for if there is not a vaccine that could prevent them. The terminal cost $\gamma_{\infty}$ illustrates the desire to keep the number of deaths low beyond the time horizon in consideration. For example, let us assume that beyond the time horizon $T$ the epidemic has been curbed and the number of new daily deaths falls below $D(T)$. In the worst case scenario we would have $D(t) = D(T)$ for $t \geq T$, hence,
\[
    \sum_{t=T}^\infty (\gamma_D)^t D(t) = D(T)\frac{\gamma_D^T}{1 - \gamma_D}.
\]
Defining $\gamma_{\infty} := \gamma_D^T/(1 - \gamma_D)$, the discounted number of deaths beyond $T$ is given by the terminal cost $\gamma_{\infty}D(T)$.

Given that the infection rate $\beta(t)$ depends on mobility, we assume that a decision maker can restrict mobility dynamically to curb the number of deaths by designing a mobility strategy $\bu(t)$ so that $\beta(t) = f(\bu(t))$. Furthermore, we assume that $\bu(t)$ is constrained to be within a set $\mathcal{U}$ reflecting that essential businesses cannot be severely restricted and that some mobility restrictions are only partially effective. 


These mobility restrictions incur a cost which could be measured in terms of a pecuniary cost to the economy, absolute number of visits lost by businesses, or impact on the utility of citizens. In our framework, we quantify the economic cost of imposing a mobility control strategy $\bu(t)$ using a cost function $C_t(\bu(t))$ which, in general, can be time-varying; hence, for example, we can use different costs for mobility restrictions on workdays and weekends. We choose to model $C_t(\bu(t))$ as a posynomial on $\bu(t)$, since this is amenable to a geometric programming approach. We investigate the problem of choosing an optimal mobility control strategy $\bu^\star(t)$ that minimizes the number of cumulative deaths while keeping the total cost of the intervention, given by $\sum_{t=0}^{T-1}C_t(\bu(t))$, below a pre-specified budget $\budget$. As we will show, this problem can be formulated as a geometric program; this stems from the fact that the states $H(t)$ and $D(t)$ can be expressed as posynomial functions of the mobility control variables $\bu(t)$, as shown in the following lemma.

\begin{lemma}\label{lem:posynomial}
The functions $H(t)$ and $D(t)$, representing the number of hospitalized individuals and deaths at time $t$, are posynomials on the entries of $\bu(t)$ for $t = 0,1,\ldots,T$.
\end{lemma}
\begin{proof}
See \cref{app:lem_posynomial_proof}.
\end{proof}

Since a positively weighted sum of posynomials is also a posynomial, we obtain our main result below.

\begin{theorem}\label{thm:GP_deaths}
The optimal mobility control strategy $\bu^\star(t)$ that minimizes the function J in \cref{eq:J} while respecting the economic budget $\budget$ and a limit on hospitalizations $\tau_H$ can be obtained by solving the following geometric program:
\begin{align}\label{GP:deaths}
    \begin{aligned}
    \underset{\bu(0),\ldots,\bu(T-1)}{\text{\emph{minimize}}} \quad & \sum_{t=1}^{T-1} \gamma_D^t D(t) + \gamma_{\infty} D(T)\\
    \text {\emph{subject to}} \quad & \sum_{t=0}^{T-1} C_t(\bu(t)) \le \budget \\
      & H(t) \leq \tau_H, \quad t = 1,\ldots,T \\
      & \bu(t) \in \mathcal{U}, \quad t = 0,\ldots,T-1
     \end{aligned}
\end{align}
where $C_t(\bu(t))$ is a posynomial cost function for all $t$, and the set of admissible control actions $\mathcal{U}$ is described by posynomial inequalities and monomial equalities.
\end{theorem}


The choice of the cost function and the budget $\budget$ is up to the decision maker and, in practice, may be difficult to choose. In the presence of available economic data, such a cost function may be computed via posynomial fitting techniques \cite{boyd2007tutorial}, or directly provided by a decision maker. In any case, we propose a principled approach to choosing a suitable $\budget$ independently of the choice of cost function $C_t(\bu(t))$. This is achieved by solving an auxiliary optimal control problem that aims to find the minimum budget required to keep the number of hospitalized individuals $H(t)$ below a threshold $\tau_H$. This auxiliary optimal control problem is also a geometric program and, hence, can be solved efficiently; this fact is stated in \cref{thm:GP_cost} and the proof follows directly from \cref{lem:posynomial}.

\begin{theorem}\label{thm:GP_cost}
The minimal budget required to keep hospitalizations below a given threshold $\tau_H$ is given by
\begin{align}
    \budget^\star = \sum_{t=0}^{T-1} C_t(\bu^\star(t)),
\end{align}
where $\bu^\star$ is the solution to the program
\begin{align}\label{GP:cost}
    \begin{aligned}
    \underset{\bu(0),\ldots,\bu(T-1)}{\text{\emph{minimize}}} \quad & \sum_{t=0}^{T-1}C_t(\bu(t)) \\
    \text{\emph{subject to}} \quad & H(t) \leq \tau_H, \quad t = 1,\ldots,T \\
      & \bu(t) \in \mathcal{U}, \quad t = 0,\ldots,T-1
     \end{aligned}
\end{align}
where $C_t(\bu(t))$ is a posynomial cost function for all $t$, and the set of admissible control actions $\mathcal{U}$ is described by posynomial inequalities and monomial equalities, which is a geometric program.
\end{theorem}

The budget $\budget^\star$ obtained from \cref{thm:GP_cost} can be seen as a conservative cost which only guarantees that hospital operations remain within capacity, avoiding overflow. The decision maker should then use a budget $\budget \geq \budget^\star$ in \cref{thm:GP_deaths} to obtain a less conservative control input $\bu(t)$.

\subsection{Control simulations}

\begin{figure}[!ht]
    \centering
    \subfloat[Value of minimal-cost control action $\bu^\star(t)$, fixed per week.]{\includegraphics[width=0.31\linewidth]{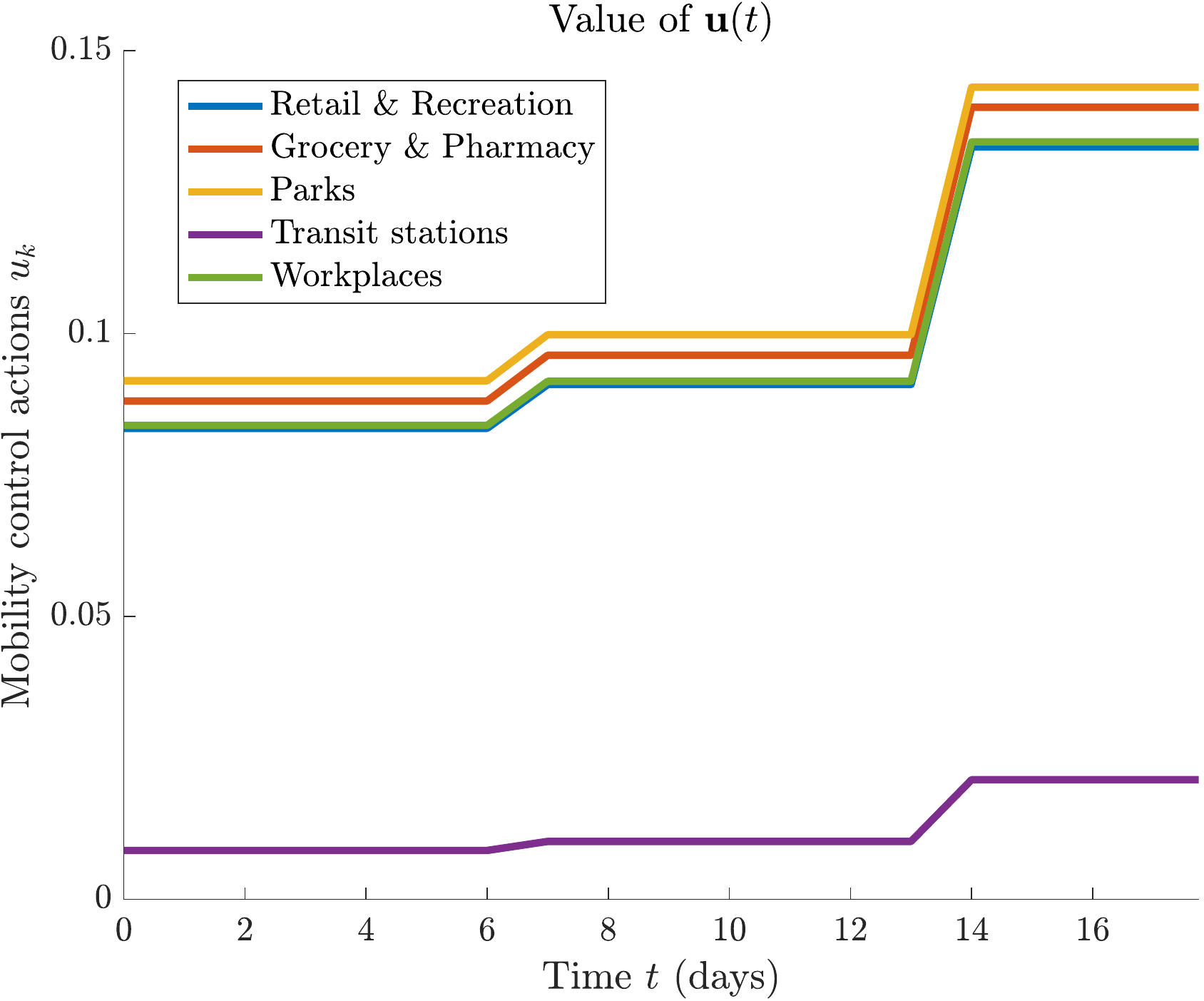}}
    \quad
    \subfloat[Actuation cost $C_t(\bu^\star(t))$ for control action $\bu^\star(t)$, fixed per week.]{\includegraphics[width=0.31\linewidth]{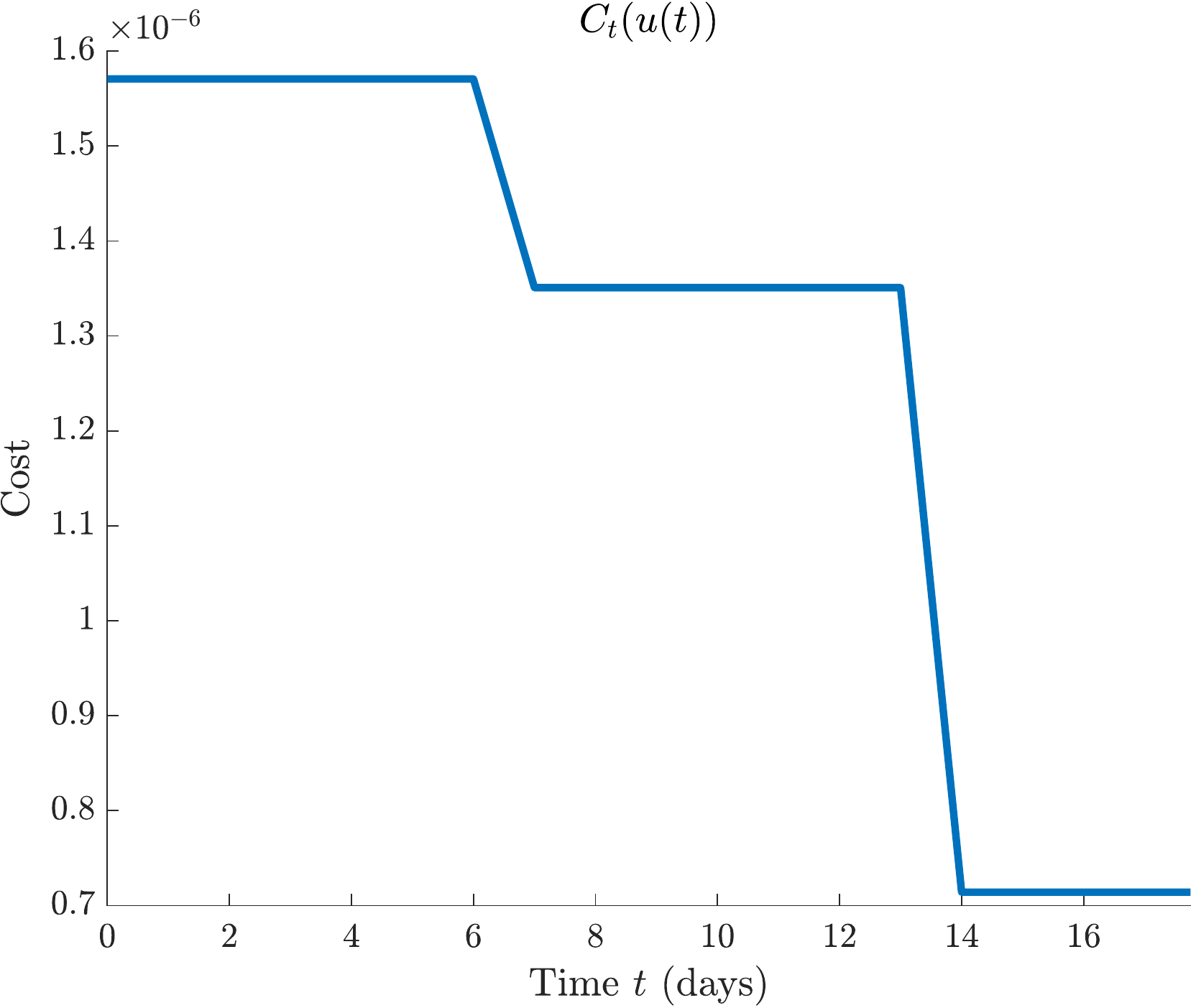}}
    \quad
    \subfloat[Number of hospitalized individuals $H(t)$; available hospital bed threshold $\tau_H$ shown as the dashed red line.]{\includegraphics[width=0.31\linewidth]{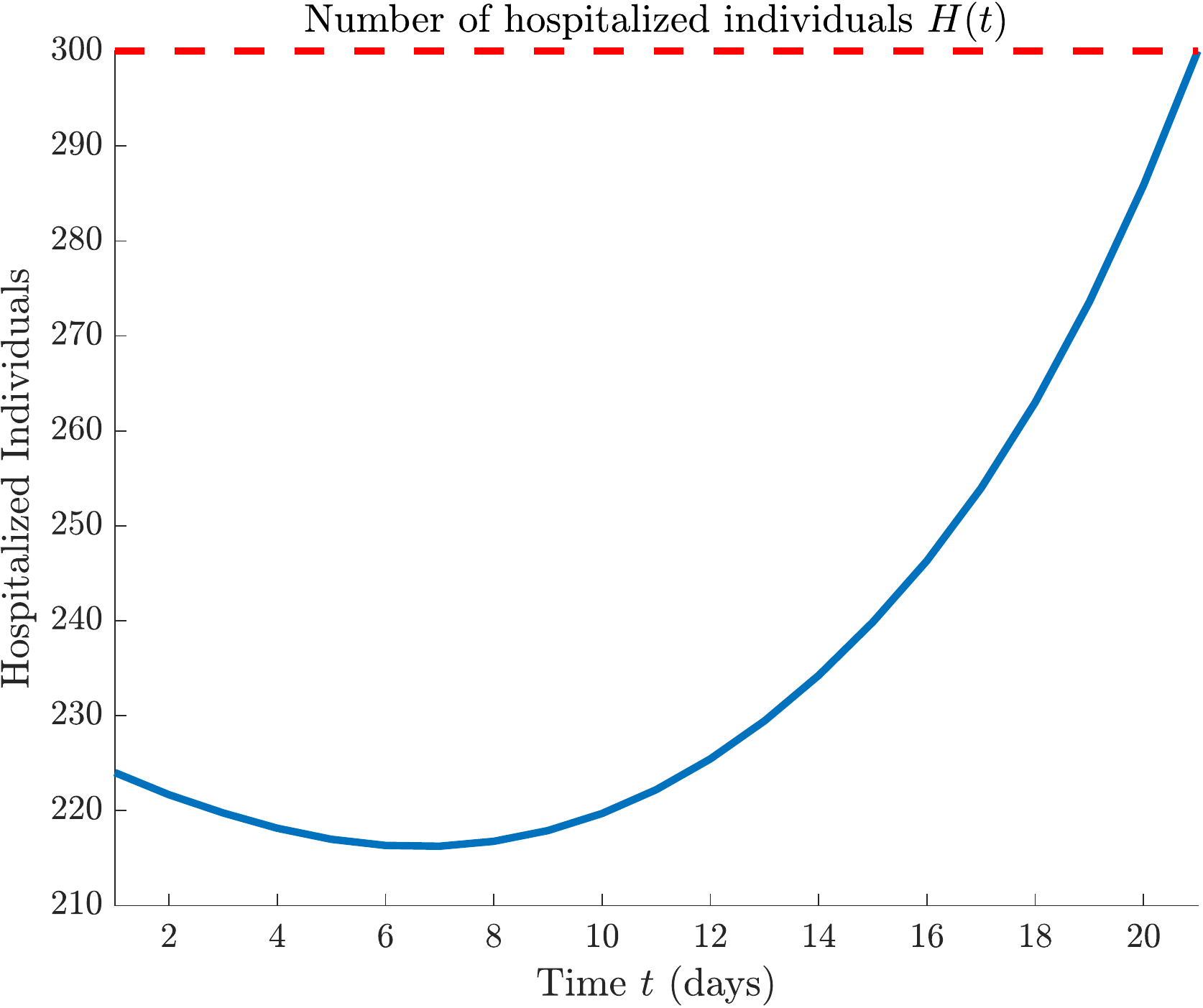}}
    \caption{Minimal-cost control strategy $\bu^\star(t)$ for Philadelphia County, PA, obtained by solving the geometric program~\cref{GP:cost}, fixing the control action to change only at the beginning of each week. The control actions for Retail \& Recreation and Workplaces are similar in Figure (a) and, hence, are overlaid. Parameters of the models used herein were learned as described in \Cref{subsec:learn_simulations}. We obtain the budget $\budget^\star$ to use in geometric program~\cref{GP:deaths} by taking the total cost of the minimal-cost control strategy, i.e., $\budget^\star = \sum_{t=0}^{T-1} C_t(\bu^\star(t))$. }
    \label{fig:GP_cost}
\end{figure}

\begin{figure}[!ht]
    \centering
    \subfloat[Value of optimal minimal-death control action $\bu^\star(t)$, fixed per week.]{\includegraphics[width=0.31\linewidth]{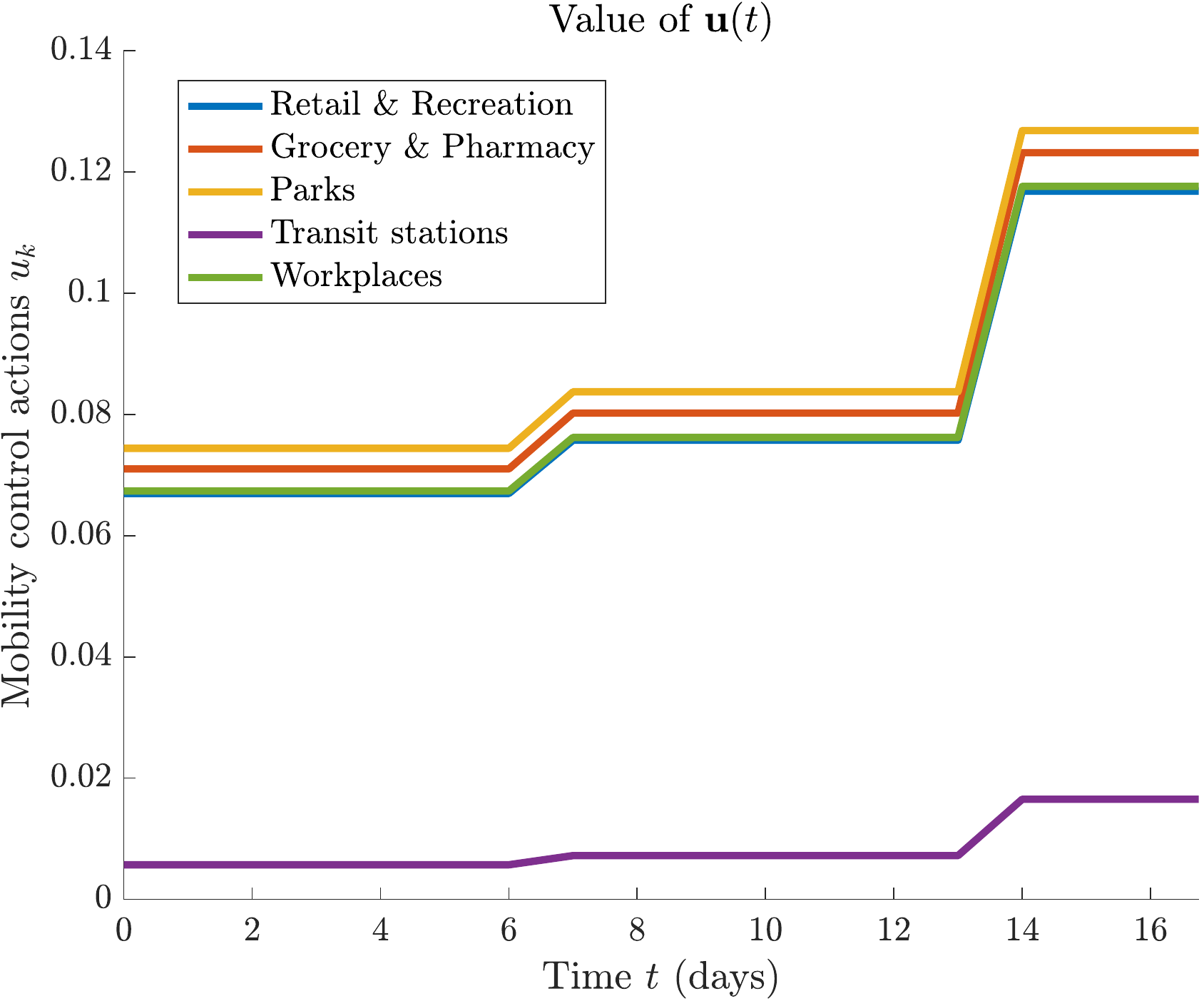}}
    \quad
    \subfloat[Actuation cost $C(\bu^\star(t))$ with control action $\bu^\star(t)$, fixed per week; average daily budget $\budget^\star/T$ shown dashed in red.]{\includegraphics[width=0.31\linewidth]{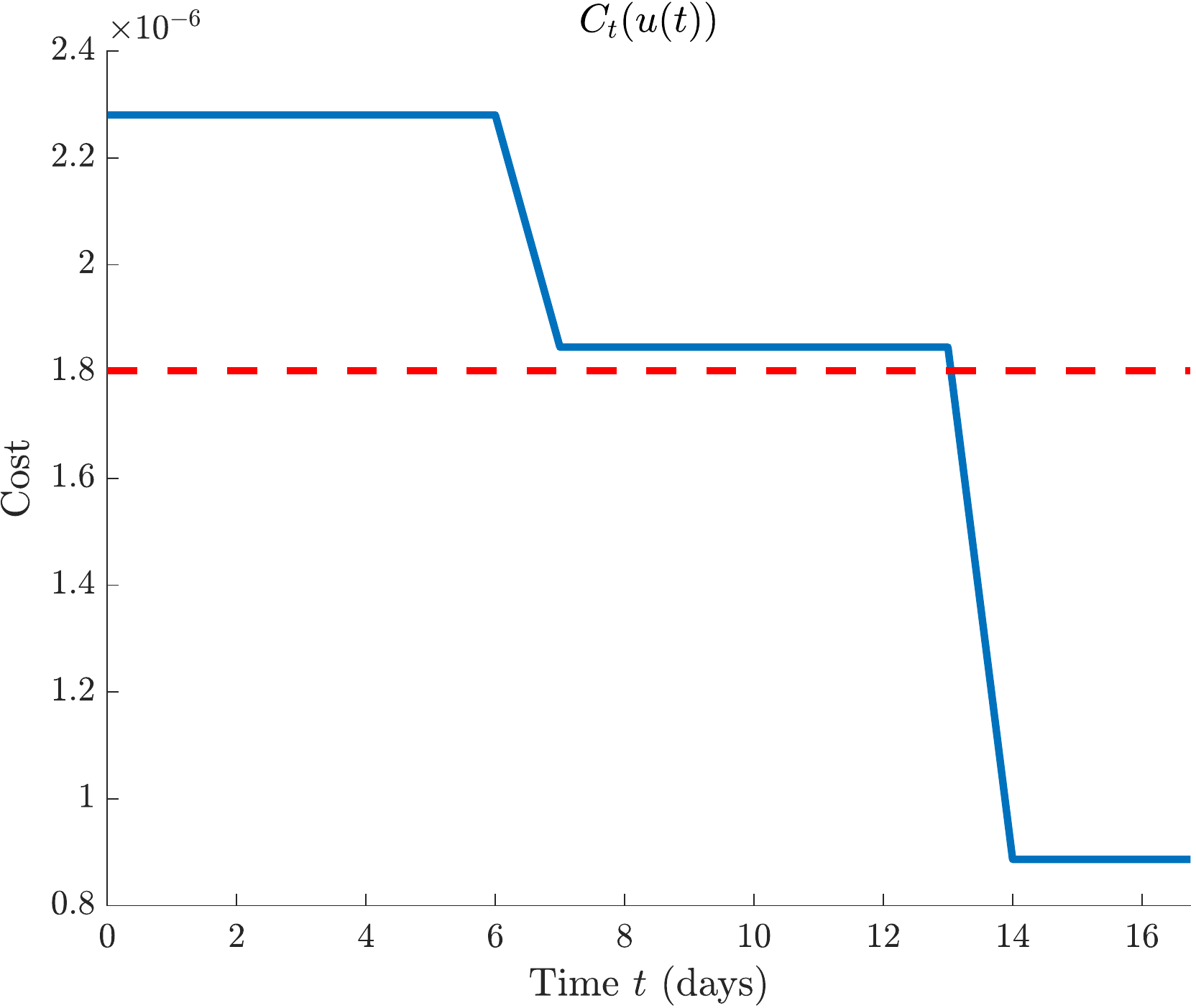}}
    \quad
    \subfloat[Number of cumulative deaths $D(t)$; baseline amount from true mobility data $\mathbf{m}(t)$ shown in red.]{\includegraphics[width=0.31\linewidth]{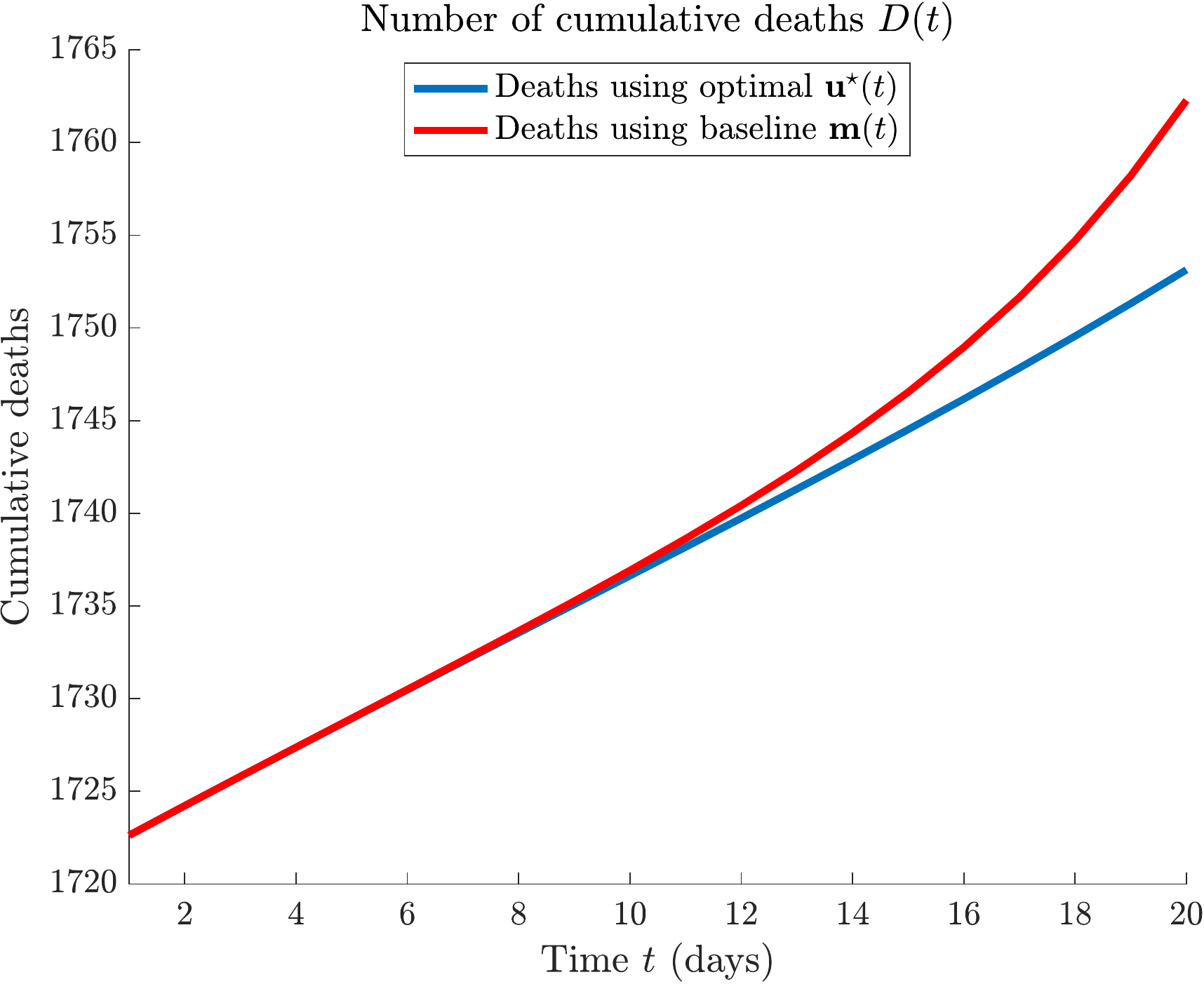}\label{fig:GP_deaths_D}}
    \caption{Optimal minimal-death control strategy $\bu^\star(t)$ for Philadelphia County, PA, obtained by solving the minimal-death GP~\cref{GP:deaths}, fixing the control action to change only at the beginning of each week from September 1st-21st, 2020. The control actions for Retail \& Recreation and Workplaces are similar in Figure (a) and, hence, are overlaid. In Figure (c) we compare the cumulative deaths predicted using the optimal control strategy $\bu^\star(t)$ (shown in blue) with the cumulative deaths predicted based on the true mobility data $\mathbf{m}(t)$ as a baseline (shown in red). Parameters of the models used herein were learned as described in \cref{subsec:learn_simulations}, with the budget $\budget^\star$ taken from the solution to the minimal-cost GP~\cref{GP:deaths}. }
    \label{fig:GP_deaths}
\end{figure}

We demonstrate the effectiveness of our control approach with a case study for the counties in the greater Philadelphia area, illustrated in \cref{fig:GP_cost} and \cref{fig:GP_deaths}. The parameters of our compartmental model as well as the mobility mapping $\beta(\bu(t))$ are learned from data as described in \Cref{subsec:learn_simulations} using the proposed multitask learning framework. Our models are trained using data from July 1st to August 31st, 2020, and tested from September 1st-21st, 2020. In particular, in \cref{fig:GP_deaths_D} we illustrate the number of cumulative deaths predicted using the optimal control strategy $\bu^\star(t)$ as compared to the number of cumulative deaths predicted based on the true mobility data $\mathbf{m}(t)$ from September 1st-21st in Philadelphia.

Since the decision maker may tune the different categories of mobility variables independently, each of the components $k = 1,\ldots,K$ of the mobility control action $\bu(t)$ may vary between a lower bound $\underline{u}_k > 0$ (representing full lockdown) and an upper bound $\overline{u}_k$ (representing no restrictions at all), leading to the set of admissible control actions actions $\mathcal{U} = \{\bu \in \mathbb{R}^K : u_k \in [\underline{u}_k, \overline{u}_k]\}$, which can be simply described by posynomial inequalities. In our simulations, we select the values $\underline{u}_k$ and $\overline{u}_k$ independently for each category based on the mobility data used in our multitask learning framework. Furthermore, as a realistic constraint we only allow the mobility control strategy to vary on a weekly basis, since rapid changes in mobility levels may not be feasible to enforce across a large region (e.g., a county).

In practice the cost function $C_t(\bu(t))$ may be supplied by a decision maker, or may be found via posynomial fitting \cite{boyd2007tutorial} using economic data from a region. In the absence of such data, we choose a time-invariant cost function $C(\bu(t))$ that satisfies the requirements of being convex in log-scale and decreasing, given by
\begin{align}\label{eq:cost_multiple}
    C(\mathbf{u}(t)) = \sum_{k=1}^K c_k\dfrac{u_k(t)^{-1} - \overline{u}_k^{-1}}{\underline{u}_k^{-1} - \overline{u}_k^{-1}},
\end{align}
where $\mathbf{c} = (c_1,\ldots,c_K)$ is a relative cost weighting of the mobility categories. This relative cost could reflect that visits lost to, for example, healthcare facilities are more costly than visits lost to retail venues. Notice that the inequality $\sum_{t=0}^{T-1}C(\bu(t)) \leq \budget$ is equivalent to $\sum_{t=0}^{T-1}\sum_{k=1}^K\frac{c_ku_k(t)^{-1}}{\underline{u}_k^{-1} - \overline{u}_k^{-1}} \leq \budget + \sum_{k=1}^K\frac{\overline{u}_k^{-1}}{\underline{u}_k^{-1} - \overline{u}_k^{-1}}$; since the left-hand side of the inequality is a posynomial in the entries of $\bu(t)$ and the right-hand side is constant, it can be readily handled by a geometric program. Similarly, minimizing $\sum_{t=0}^{T-1}C(\bu(t))$ is equivalent to minimizing $\sum_{t=0}^{T-1}\sum_{k=1}^K\frac{c_ku_k(t)^{-1}}{\underline{u}_k^{-1} - \overline{u}_k^{-1}}$, which is again a posynomial function amenable to geometric programming.

\section{Conclusion and Future Work} \label{sec:conclusion}

In this paper we present a data-driven learning and optimal control framework that aims to bridge the gap between optimal control theory of epidemic models and applicable data-driven models for analyzing the spread of COVID-19. To identify the parameters of our model, we propose a multitask learning approach that leverages mobility and epidemic data from multiple regions to capture how daily changes in mobility patterns affect the spread of the disease, and to accurately predict the resulting daily and cumulative deaths. Using this data-driven model we present an optimal control framework using geometric programming to efficiently design non-pharmaceutical interventions to limit the spread of the epidemic while obeying a budget constraint on the economic loss incurred. Furthermore, we present a principled method for determining such a budget based on eliminating excess deaths due to over-utilization of hospital resources. We validate both our model and our control framework in a case study on the greater Philadelphia area.

In the future, this work could be extended to accommodate for robustness considerations as well as stochastic transitions in the epidemic layer, which introduce the additional challenge of expressing chance constraints as posynomial functions. The success of geometric programming in our work comes from expressing states of the system as posynomials on the mobility variables, allowing for the potential extension to models with more sophisticated mappings from human-mobility to epidemic dynamics. For example, generalized geometric programming admits functions that are max-monomials or posynomials with fractional exponents \cite{boyd2007tutorial}, which opens the door to modeling epidemic dynamics and human-mobility using ReLU neural networks or posynomial approximations to arbitrary functions. Furthermore, since geometric programs can be solved efficiently, our approach could be applied to models with higher dimensionality; for example, networked metapopulation models which can make use of more granular datasets.

\appendix

\section{Proof of \Cref{lem:posynomial}}\label{app:lem_posynomial_proof}
We can rewrite equations \cref{eq:dyn_1}-\cref{eq:dyn_4} in matrix form by defining a state vector $\mathbf{y}(t) = \left[E\left(t\right) ,  I\left(t\right), A\left(t\right),  H\left(t\right)\right]^{\intercal}$ to obtain the dynamics
\begin{align}
    \mathbf{y}(t+1) &=\left[\begin{array}{cccc}
\!1-\rho_{EI}-\rho_{EA} & S_{0}\beta\left(\bu(t)\right) & \gamma_{A}S_{0}\beta\left(\bu(t)\right) & 0\\
\rho_{EI} & 1-\rho_{IR}-\rho_{IH} & 0 & 0\\
\rho_{EA} & 0 & 1-\rho_{AR} & 0\\
0 & \rho_{IH} & 0 & 1-\rho_{HR}-\rho_{HD}\!
\end{array}\right]\!\mathbf{y}\left(t\right) \nonumber \\
&=: M_{t}\mathbf{y}(t). \label{eq:matrices}
\end{align}
It follows that 
\begin{align*}
    H\left(t\right)=\left[
0, 0, 0, 1\right]M_{t-1}\cdots M_{1}M_{0}\left[
E\left(0\right), I\left(0\right), A\left(0\right), H\left(0\right)\right]^{\intercal}= f_{H}^{t}\left(\left\{ u\left(k\right)\right\} _{k=0}^{t-2}\right).
\end{align*}
We assume that $\rho_{EI} < 1, \rho_{IR} + \rho_{IH} < 1, \rho_{AR} < 1$, and $\rho_{HR} + \rho_{HD} < 1$, which is reasonable because these numbers capture the fraction of agents in each compartment that transition to other compartments, which must be less than one for the model to be meaningful. Recalling that the mobility mapping $\beta(\bu(t))$ is a posynomial, each of the matrices $M_t$ have posynomial entries on $\beta(\bu(t))$, and thus, on $\bu(t)$; thus, it follows that the function $f_{H}^{t}\left(\left\{ u\left(k\right)\right\} _{k=0}^{t-2}\right)$ is also a posynomial on $\bu(t)$ as it is the product of matrices with entries that are posynomials on $\bu(t)$. Moreover, from \cref{eq:D_t} we can see that $D(t)$ is simply a sum of positive constants and positive scalar multiples of $H(t)$ and is also a posynomial.



\bibliographystyle{plain}
\bibliography{references,gp_and_data}

\end{document}